\documentclass[graybox]{svmult}

\usepackage{type1cm}        
\usepackage[bottom]{footmisc}

\usepackage{amsmath,amsfonts,amssymb}
\usepackage{enumitem}
\usepackage{tikz}
\usepackage{array}
\usepackage{makecell}
\usepackage{algorithm,algorithmic}
\newcolumntype{x}[1]{>{\centering\arraybackslash}p{#1}}
\usepackage{tikz}

\usepackage{diagbox}

\makeindex             

\begin{document}
	
	\title*{An overlapping waveform-relaxation preconditioner for economic optimal control problems with state constraints}
	\titlerunning{Overlapping WR preconditioner for economic control with state constaints}
	\author{Gabriele Ciaramella and Luca Mechelli}
	\authorrunning{G. Ciaramella and L. Mechelli}
	\institute{Gabriele Ciaramella \at Politecnico di Milano \email{gabriele.ciaramella@polimi.it}
		\and Luca Mechelli \at Universit\"at Konstanz \email{luca.mechelli@uni-konstanz.de}}
	%
	%
	\maketitle
	
	\abstract*{In this work, a class of parabolic economic optimal control problems is considered. 
		These problems are characterized by pointwise state constraints regularized by a parameter, 
		which transforms the pure state constraints in mixed control-state ones.
		However, the convergence of classical (semismooth) Newton methods deteriorates 
		for decreasing values of the regularization parameter. To tackle this problem, a nonlinear 
		preconditioner is introduced. This is based on an overlapping optimized waveform-relaxation method 
		characterized by Robin transmission conditions. Numerical experiments show that appropriate 
		choices of the overlap and of the Robin parameter lead to a preconditioned Newton method with a robust convergence 
		against the state constraints regularization parameter.}

	\section{Introduction} 
	This work is concerned with the numerical solution of so-called economic optimal control problems of the parabolic type. Let $\Omega=(-1,1)$, $T>0$ and $\mathcal{U}:= L^2(0,T;L^2(\Omega))$ endowed with its norm $\|\cdot\|_\mathcal{U}$. We want to solve
	\begin{subequations}
	\label{Mechelli_mini_13_P}
	\begin{equation}
	\label{Mechelli_mini_13_cost_functional}
	\min_{\mathcal{U}\times\mathcal{U}} \mathcal{J}(u,w) := \frac{1}{2}\|u\|^2_\mathcal{U} + \frac{1}{2}\|w\|^2_\mathcal{U},
	\end{equation}
	subject to the PDE-constraint
	\begin{equation}
	\label{Mechelli_mini_13_state_equation}
	\begin{aligned}
	y_t(t,x)-\Delta y(t,x) & = f(t,x)+u(t,x), && \text{in } (0,T)\times\Omega, \\
	y(t,-1) = y(t,1) & = 0,  && \text{in } (0,T), \\
	y(0,x) & = y_\circ(x),  && \text{in } \Omega,
	\end{aligned}
	\end{equation}
	with $y_\circ\in L^2(\Omega)$ and $f\in\mathcal{U}$, and to mixed control-state constraints
	\begin{equation}
	\label{Mechelli_mini_13_mixed_constraints}
	\left|u(t,x)\right| \leq c_u, \quad \left|y(t,x)+\varepsilon w(t,x)\right|\leq c_y(t), \quad \text{in } (0,T)\times\Omega,
	\end{equation}
	\end{subequations}
	where $c_u,\varepsilon>0$ and $c_y\in L^2(0,T)$ with $c_y(t)>0$ for $t\in(0,T)$. 
	Problem \eqref{Mechelli_mini_13_P} is related to the virtual control approach \cite{Mechelli_mini_13_Krumbiegel2009, Mechelli_mini_13_Mechelli2019, Mechelli_mini_13_Mechelli2018}, which is a regularization technique for pointwise state-constrained problems. Under further assumptions on $w$, in fact, one can show that, as $\varepsilon\to 0$, the solution to \eqref{Mechelli_mini_13_P} converges to the one of the same optimal control problem with \eqref{Mechelli_mini_13_mixed_constraints} replaced by $\left|u(t,x)\right| \leq c_u$ and $\left|y(t,x)\right|\leq c_y(t)$ in $(0,T)\times\Omega$; see, e.g., \cite{Mechelli_mini_13_Mechelli2019}. Note that there are no weights in front of the control norms in \eqref{Mechelli_mini_13_cost_functional}. This is because of the regularization parameter $\varepsilon$, which is also used to tune the magnitude of the controls $u$ and $w$. For example, the smaller is $\varepsilon$, the larger is $\|w\|_\mathcal{U}$. In contrast to classical optimal control problems, where the goal is to reach a precise target configuration, the focus of \eqref{Mechelli_mini_13_P} is to find minimum-energy feasible controls such that the state solution to \eqref{Mechelli_mini_13_state_equation} satisfies the bounds \eqref{Mechelli_mini_13_mixed_constraints}. This difference is particularly evident in the cost functional $\mathcal{J}$ in (1a), where only the norm squared of the controls are considered, instead of typical tracking-type terms. For these reasons, problems of the type (1) are called economic optimal control problems.
	A typical example is the optimal heating and cooling of residual buildings \cite{Mechelli_mini_13_Mechelli2019}. Note that, for any given $u \in \mathcal{U}$, the state equation \eqref{Mechelli_mini_13_state_equation} admits a unique (weak) solution $y=y(u)\in W(0,T):= \left\{\varphi\in L^2(0,T;H^1(\Omega))\big| \varphi_t \in L^2(0,T;H^{-1}(\Omega))\right\}$; see, e.g., \cite{Mechelli_mini_13_Troeltzsch2010,Mechelli_mini_13_Mechelli2018}. We assume that the admissible set $\mathcal{U}^\varepsilon_\mathsf{ad}$ has non-empty interior, where	$\mathcal{U}^\varepsilon_\mathsf{ad} := \left\{(u,w)\in \mathcal{U}\times \mathcal{U}\big| u \text{ and } y(u)+\varepsilon w \text{ satisfies \eqref{Mechelli_mini_13_mixed_constraints}} \right\}\subset \mathcal{U}\times\mathcal{U}$. This guarantees that \eqref{Mechelli_mini_13_P} admits a unique solution $(\bar u,\bar w)\in \mathcal{U}^\varepsilon_\mathsf{ad}$ \cite{Mechelli_mini_13_Troeltzsch2010}.
	The first-order necessary and sufficient optimality system \cite{Mechelli_mini_13_Mechelli2018,Mechelli_mini_13_Troeltzsch2010} of problem \eqref{Mechelli_mini_13_P} is
	\begin{equation}
	\label{Mechelli_mini_13_opt_syst}
	\begin{aligned}
	y_t(t,x)-\Delta y(t,x) &= \mathcal P(q(t,x))+f(t,x), && \text{in } (0,T)\times\Omega, \\
	y(t,-1) = y(t,1) &= 0,  && \text{in } (0,T),  \\
	y(0,x) & = y_\circ(x),  && \text{in } \Omega, \\
	q_t(t,x)+\Delta q(t,x) &= \mathcal{Q}^\varepsilon(y(t,x)), && \text{in } (0,T)\times\Omega, \\
	q(t,-1) = q(t,1)& = 0,  && \text{in } (0,T),  \\
	q(T,x) & = 0,  && \text{in } \Omega, \\
	\end{aligned}
	\end{equation}
     where $\mathcal{Q}^\varepsilon(y(t,x)):= \frac{1}{\varepsilon^2}(\max\{y(t,x)-c_y(t),0\}+\min\{y(t,x) + c_y(t),0\})$ and $\mathcal{P}(q(t,x)):= \max\{-c_u,\min\{c_u,q(t,x)\}\}$, for all $(t,x)\in(0,T)\times\Omega$, with $q$ the so-called adjoint variable. The pair $(\bar y,\bar q)$ is the solution to \eqref{Mechelli_mini_13_opt_syst} if and only if $(\bar u(t,x),\bar w(t,x)) = \left(\mathcal P(\bar q(t,x)),-\varepsilon\mathcal{Q}^\varepsilon(\bar y(t,x))\right)$, for $(t,x)\in(0,T)\times\Omega$, is the optimal solution to \eqref{Mechelli_mini_13_P}. System \eqref{Mechelli_mini_13_opt_syst} can be rewritten in the form
     \begin{equation}
     \label{Mechelli_mini_13_Newton_system}
     \mathcal{F}(y,q) = 0
     \end{equation}
     and thus solved by using a semismooth Newton method; see, e.g., \cite{Mechelli_mini_13_Mechelli2018,Mechelli_mini_13_Hintermueller2002}. 
     
     As shown in \cite{Mechelli_mini_13_Mechelli2019}, the semismooth Newton method lacks of convergence if the parameter $\varepsilon$ is not sufficiently large. This is, however, in contrast with typical applications, where a sufficiently small $\varepsilon$ is required \cite{Mechelli_mini_13_Mechelli2019,Mechelli_mini_13_Krumbiegel2009}. The goal of this paper is to tackle this problem by using a nonlinear preconditioning technique based on an overlapping optimized waveform-relaxation method (WRM) characterized by Robin transmission conditions \cite{Mechelli_mini_13_Dolean2016,Mechelli_mini_13_GanderHalpern1}. To the best of our knowledge, nonlinear preconditioning techniques have never been used for economic control problems. Therefore, this work aims to provide a first concrete study in order to show the applicability of WRM-based nonlinear preconditioners for this class of optimization problems. In particular, our goal is to assess the convergence behavior of the WRM nonlinear preconditioned Newton and its robustness against the regularization parameter $\varepsilon$. Our studies show that appropriate 
     choices of the overlap $L$ and of the Robin parameter $p$ lead to a preconditioned Newton method with a robust convergence with respect to $\varepsilon$. Let us also mention that for elliptic optimal control problems, it is possible to consider different transmission conditions; see, e.g., \cite{Mechelli_mini_13_Benamou1996,Mechelli_mini_13_Heink2006}.
     
      The paper is organized as follows. In Section~\ref{Mechelli_mini_13_sec2}, we introduce the WRM and present the algorithm for the proposed preconditioned generalized Newton. In Section~\ref{Mechelli_mini_13_sec3}, we report two numerical experiments that show the convergence behavior of the proposed computational framework in relation of the parameters characterizing problem \eqref{Mechelli_mini_13_P} and the optimized WRM.
     
     \section{The waveform-relation and the preconditioned generalized Newton methods}
     \label{Mechelli_mini_13_sec2}
     Let $\Omega$ be decomposed into two overlapping subdomains $\Omega_1= (-1,L)$ and $\Omega_2= (-L,1)$, where $2L\in(0,1)$ is the size of the overlap. Moreover, let $p>0$ and consider the operator $\mathcal{R}_j$ defined as $\mathcal{R}_j(y):= y_x+(-1)^{3-j}py$ for $j=1,2$. The WRM consists in iteratively solving, for $n\in\mathbb{N}$, $n\geq 1$, the system 
     \begin{equation}
     \label{Mechelli_mini_13_WRM-subproblems}
     \begin{aligned}
     y_t^{j,n}(t,x)-\Delta y^{j,n}(t,x) &= \mathcal{P}(q^{j,n}(t,x))+f(t,x), && \text{in } (0,T)\times\Omega_j, \\
     y^{j,n}(t,(-1)^j) &= 0,  && \text{in } (0,T),  \\
     \mathcal{R}_j(y^{j,n})(t,(-1)^{3-j}L) & = \mathcal{R}_j(y^{3-j,n-1})(t,(-1)^{3-j}L), &&  \text{in } (0,T), \\
     y^{j,n}(0,x) & = y_\circ(x),  && \text{in } \Omega_j, \\
     q^{j,n}_t(t,x)+\Delta q^{j,n}(t,x) &= \mathcal{Q}^\varepsilon(y^{j,n}(t,x)),  && \text{in } (0,T)\times\Omega_j, \\
     q^{j,n}(t,(-1)^j) & = 0,  && \text{in } (0,T),  \\
     \mathcal{R}_j(q^{j,n})(t,(-1)^{3-j}L) & = \mathcal{R}_j(q^{3-j,n-1})(t,(-1)^{3-j}L), && \text{in } (0,T), \\
     q^{j,n}(T,x) & = 0,  && \text{in } \Omega_j, \\
     \end{aligned}
     \end{equation}
     for $j=1,2$. We show first the well-posedness of the method.
     \begin{theorem}
     \label{Mechelli_mini_13_thm:wellpos}
     Let $g^1_y,g^2_y,g^1_q,g^2_q\in H^{1/4}(0,T)$ be initialization functions for the WRM, i.e., $\mathcal{R}_j(y^{j,1})(t,(-1)^{3-j}L)=g^j_y(t)$ and $\mathcal{R}_j(q^{j,1})(t,(-1)^{3-j}L)=g^j_q(t)$ for $t\in(0,T)$,
     with compatibility conditions $g^j_y(0) = \mathcal{R}_j(y_\circ)(t,(-1)^{3-j}L)$ and $g^j_q(0)=0$ for $j=1,2$. Then the WRM \eqref{Mechelli_mini_13_WRM-subproblems} is well-posed.
     \end{theorem}
 	\begin{proof}
 	For $j=1,2$, we define $H^{2,1}_j:= L^2(0,T;H^2(\Omega_j))\times H^1(0,T;L^2(\Omega_j))$ and $\mathcal{U}_j= L^2(0,T;L^2(\Omega_j))$. Consider the auxiliary problems
 	\[
 	\min_{\mathcal{U}_j\times\mathcal{U}_j} \mathcal{J}_\text{aux}(u^j,w^j)= \frac{1}{2}\|u^j\|_{\mathcal{U}_j}^2+\frac{1}{2}\|w^j\|_{\mathcal{U}_j}^2+\int_0^T g^j_q(t)y^{j}(t,(-1)^{3-j}L)\mathrm{d} t
 	\]
 	subject to 
 	\[
 	\begin{aligned}
 	y^j_t(t,x)-\Delta y^j(t,x) & = u^j(t,x)+f(t,x), && \text{in } (0,T)\times\Omega_j, \\
 	y^j(t,(-1)^j) & = 0,  && \text{in } (0,T), \\
 	\mathcal{R}_j(y^{j})(t,(-1)^{3-j}L) & = g^j_y(t)&& \text{in } (0,T), \\
 	y^j(0,x) & = y_\circ(x),  && \text{in } \Omega_j, \\
 	\left|u^j(t,x)\right| \leq c_u, \quad \left|y^j(t,x)\right.&\left.+\varepsilon w^j(t,x)\right|\leq c_y(t), && \text{in } (0,T)\times\Omega_j,
 	  %
 	\end{aligned}
 	\]
 	for given $g^j_y,g^j_q\in H^{1/4}(0,T)$. These auxiliary optimal control problems admit a unique optimal solution $(\bar u^j,\bar w^j)\in\mathcal{U}_j\times\mathcal{U}_j$ for $j=1,2$. Furthermore, the optimality system corresponding to each problem has the form of \eqref{Mechelli_mini_13_WRM-subproblems} and it is uniquely solvable by $(\bar y^j,\bar q^j)\in H^{2,1}_j\times H^{2,1}_j$ such that
   \[
   (\bar u^j(t,x),\bar w^j(t,x)) = (\mathcal P(\bar q^j(t,x)),-\varepsilon\mathcal{Q}^\varepsilon(\bar y^j(t,x))), \quad \text{in } (0,T)\times\Omega_j.
   \]
   For more details see \cite{Mechelli_mini_13_Troeltzsch2010,Mechelli_mini_13_LionsMagenes1,Mechelli_mini_13_GanderHalpern1}.
   This proves well-posedness of the WRM for $n=1$ and $j=1,2$. By iteratively applying the previous arguments is then easy to show that the WRM is well-posed for $n>1$, because $y^{j,1}((-1)^jL),y_x^{j,1}((-1)^jL),q^{j,1}((-1)^jL),q_x^{j,1}((-1)^jL)\in L^2(0,T)$.
 	\end{proof}
 Theorem~\ref{Mechelli_mini_13_thm:wellpos} implies that \eqref{Mechelli_mini_13_WRM-subproblems} admits a unique solution $(y^{j,n},p^{j,n})\in H^{2,1}_j\times H^{2,1}_j $ for $j=1,2$ and $n\geq1$. Note that, at each iteration of the WRM, the solution at iteration $n$ depends on the one at iteration $n-1$. Therefore, we can define the solution mappings $\mathcal S_j: H^{2,1}_{3-j}\times H^{2,1}_{3-j} \to H^{2,1}_j\times H^{2,1}_{j}$ for $j=1,2$ as
 \begin{equation}
 \label{Mechelli_mini_13_solutionmaps}
 \begin{aligned}
 (y^1,q^1) = \mathcal S_1(y^2,q^2) \text{ solves \eqref{Mechelli_mini_13_WRM-subproblems} for } j=1 \text{, } y^{2,n-1}=y^2 \text{ and } q^{2,n-1}=q^2, \\
 (y^2,q^2) = \mathcal S_2(y^1,q^1) \text{ solves \eqref{Mechelli_mini_13_WRM-subproblems} for } j=2 \text{, } y^{1,n-1}=y^1 \text{ and } q^{1,n-1}=q^1,
 \end{aligned}
 \end{equation}
 and the preconditioned form of \eqref{Mechelli_mini_13_Newton_system} as
 \begin{equation}
 \label{Mechelli_mini_13_precond_semismooth_newton}
 \mathcal{F}_P(y^1,q^1,y^2,q^2) = (\mathcal{F}_1(y^1,q^1,y^2,q^2),\mathcal{F}_2(y^1,q^1,y^2,q^2)) = 0,
 \end{equation}  
 where $\mathcal{F}_j(y^1,q^1,y^2,q^2)= (y^j,q^j)-\mathcal{S}_j(y^{3-j},q^{3-j})$, for $j=1,2$.
 To solve \eqref{Mechelli_mini_13_precond_semismooth_newton}, we apply a generalized Newton method. To do so, we assume that the maps $\mathcal{S}_j$, $j=1,2$, admit derivative\footnote{Since the functions $\mathcal{S}_j$ are implicit functions of semismooth functions, one cannot directly invoke the implicit function theorem to obtain the desired regularity. Hence, investigating the existence and regularity of $D\mathcal{S}_j$ requires a detailed theoretical analysis, which is beyond the scope of this short manuscript.} $D\mathcal{S}_j$.
 This allows us to characterize the derivative $D\mathcal{F}_P$ and its application to a direction $\mathbf{d}^{3-j}=(d^{3-j}_y,d^{3-j}_q)\in H^{2,1}_{3-j}\times H^{2,1}_{3-j}$, which is needed for the generalized Newton method. Let $z^j:= (y^j,q^j)\in H^{2,1}_j\times H^{2,1}_j$ for $j=1,2$. Thus, we have that $z^j= \mathcal{S}_j(z^{3-j})$, according to the definition of the mapping $\mathcal{S}_j$ in \eqref{Mechelli_mini_13_solutionmaps}. Moreover, we have that $\mathcal F_j(\mathcal{S}_j(z^{3-j}),z^{3-j})=0$. From this we formally obtain
\[
D_1 \mathcal F_j(\mathcal{S}_j(z^{3-j}),z^{3-j})D\mathcal{S}_j(z^{3-j})(\mathbf{d}^{3-j})+D_2\mathcal F_j(\mathcal{S}_j(z^{3-j}),z^{3-j})(\mathbf{d}^{3-j}) = 0,
\]
which leads to $D\mathcal{S}_j(y^{3-j},q^{3-j})(\mathbf{d}^{3-j})=(\widetilde{y}^j,\widetilde{q}^j)$ where $ (\widetilde{y}^j,\widetilde{q}^j)$ solves
\begin{equation}
\label{Mechelli_mini_13_linearized_WRM-subproblems}
\begin{aligned}
\widetilde y_t^{j}(t,x)-\Delta \widetilde y^{j}(t,x) &= \widetilde q^{j}(t,x)\chi_{\mathcal{I}(q^j)}(t,x), && \text{in } (0,T)\times\Omega_j, \\
\widetilde y^{j}(t,(-1)^j) &= 0,  && \text{in } (0,T),  \\
\mathcal{R}_j(\widetilde y^{j})(t,(-1)^{3-j}L) & =  \mathcal{R}_j(d^{3-j}_y) (t,(-1)^{3-j}L), &&  \text{in } (0,T), \\
\widetilde y^{j}(0,x) & = 0,  && \text{in } \Omega_j, \\
\widetilde q^{j}_t(t,x)+\Delta \widetilde q^{j}(t,x) &= \frac{\widetilde y^{j}(t,x)}{\varepsilon^2}\chi_{\mathcal{A}(y^j)}(t,x),  && \text{in } (0,T)\times\Omega_j, \\
\widetilde q^{j}(t,(-1)^j) & = 0,  && \text{in } (0,T),  \\
\mathcal{R}_j(\widetilde q^{j})(t,(-1)^{3-j}L) & = \mathcal{R}_j(d^{3-j}_q)(t,(-1)^{3-j}L), && \text{in } (0,T), \\
\widetilde q^{j,n}(T,x) & = 0,  && \text{in } \Omega_j. \\
\end{aligned}
\end{equation}
for $j=1,2$, with $\chi_{\mathcal{I}(q^j)}$ and $\chi_{\mathcal{A}(y^j)}$ the characteristic functions of the sets
\[
\begin{aligned}
\mathcal{I}(q^j) & := \{(t,x)\in(0,T)\times\Omega_j\big| \left|q^j(t,x)\right|\leq c_u \}, \\
\mathcal{A}(y^j)&:= \{(t,x)\in(0,T)\times\Omega_j\big| \left|y^j(t,x)\right| > c_y(t)\}.
\end{aligned}
\]
Note that \eqref{Mechelli_mini_13_linearized_WRM-subproblems} is a linearization of the WRM subproblems \eqref{Mechelli_mini_13_WRM-subproblems}. Now, we can resume our preconditioned generalized Newton method in Algorithm~\ref{Mechelli_mini_13_Alg1}.
\begin{algorithm}
	\begin{algorithmic}[1]
		\STATE \textbf{Data:} Initial guess $y^{j,0}$ and $q^{j,0}$ for $j=1,2$, tolerance $\tau$.
		\STATE Perform one WRM step to compute $\mathcal{S}_j(y^{3-j,0},q^{3-j,0})$;
		\STATE Assemble $\mathcal{F}_P(y^{1,0},q^{1,0},y^{2,0},q^{2,0})$ and set $k=0$;
		\WHILE {$\|\mathcal{F}_P(y^{1,k},q^{1,k},y^{2,k},q^{2,k})\|\geq \tau$} {
			\STATE Compute $\mathbf{d}^1,\mathbf{d}^2$ solving $D\mathcal F_P(y^1,q^1,y^2,q^2)(\mathbf{d}^1,\mathbf{d}^2) = -\mathcal F_P(y^1,q^1,y^2,q^2)$ by using a matrix-free Krylov method, e.g., GMRES, and considering that $D\mathcal F_P(y^1,q^1,y^2,q^2)(\mathbf{d}^1,\mathbf{d}^2)= (\mathbf{d}^1-(\widetilde{y}^1,\widetilde{q}^1),\mathbf{d}^2-(\widetilde{y}^2,\widetilde{q}^2))$, with $(\widetilde{y}^j,\widetilde{q}^j)$ solution to the linearized subproblems \eqref{Mechelli_mini_13_linearized_WRM-subproblems} for $j=1,2$;
			\STATE Update $(y^{j,k+1},q^{j,k+1})=(y^{j,k},q^{j,k})+\mathbf{d}^j$ and set $k=k+1$;
			\STATE Perform one WRM step to compute $\mathcal{S}_j(y^{3-j,k},q^{3-j,k})$;
			\STATE Assemble $\mathcal{F}_P(y^{1,k},q^{1,k},y^{2,k},q^{2,k})$;
		}
		\ENDWHILE
	\end{algorithmic}
	\caption{WRM-preconditioned generalized Newton method \label{Mechelli_mini_13_Alg1}}
\end{algorithm}
	\vspace{-3.5em}
\section{Numerical experiments}\label{Mechelli_mini_13_sec3}
\begin{figure}[t]
	\centering 
	\includegraphics[height=30mm]{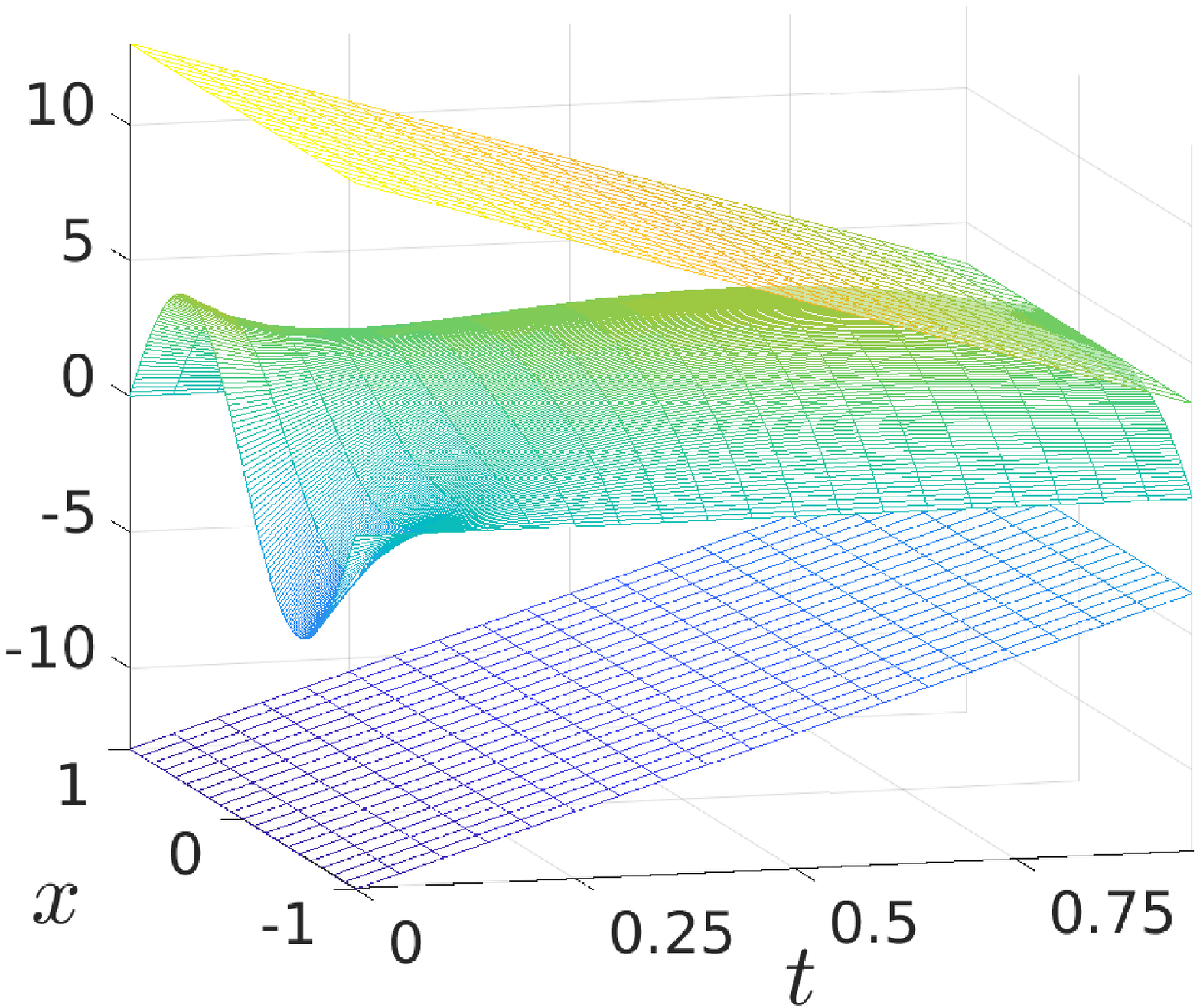} \hspace{1em}
	\includegraphics[height=30mm]{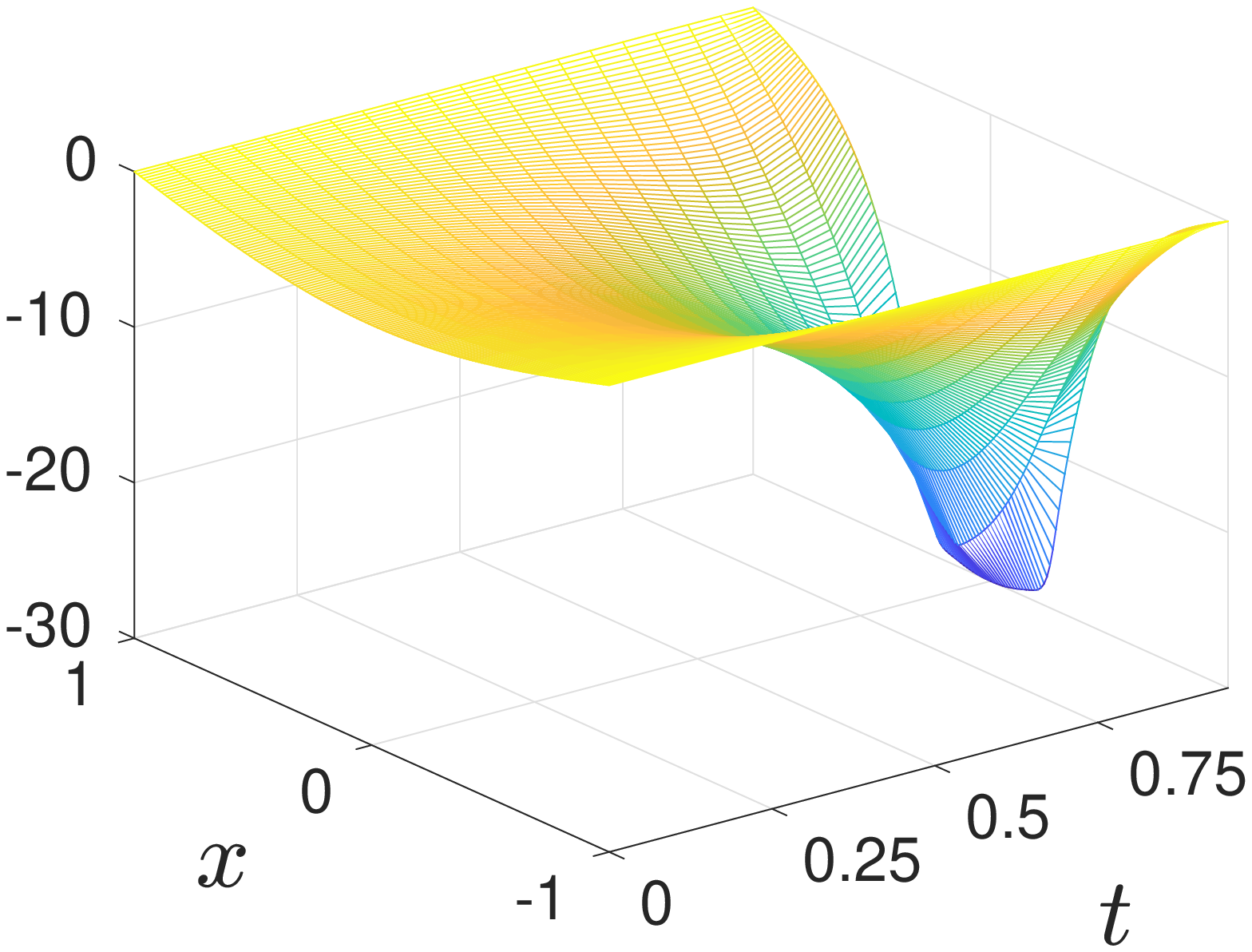}
	\caption{Test1: Optimal state with bound $c_y$ (left) and control (right) for $\varepsilon=5\times 10^{-4}$.} \label{Mechelli_mini_13_Fig1}
\end{figure}
\begin{table}[t]
	\centering
	\begin{tabular}{|c|c|c|c|c|c|c|c|}
		\hline
		$L$ & \diagbox{\hspace{1em}$p$}{$\varepsilon$\hspace{1em}} & $10^{-1}$ &$5\times 10^{-2}$ & $10^{-2}$ & $5\times 10^{-3}$ &$10^{-3}$ & $5\times 10^{-4}$ \\ \hline \multicolumn{8}{c}{}\\[-0.5em]\hline 
		$\Delta x$ & $10^{-6}$              &  $4 (5$--$2)$ & $4(6$--$2)$ & $5 (12$--$2)$& $6(13$--$2)$  &$7(35$--$2)$ & $8(45$--$2)$\\ \hline
		$\Delta x$ &  $10^{-4}$              &  $4 (5$--$2)$ &$4(6$--$2)$ & $5 (13$--$2)$& $6(13$--$2)$& $7(34$--$2)$& $8(45$--$2)$\\ \hline
		$\Delta x$ &  $10^{-2}$            &   $4 (6$--$2)$&$4(6$--$2)$ & $ 5 (11$--$2)$  & $ 6 (13$--$2)$& $7(30$--$2)$ & $8(43$--$2)$\\ \hline
		$\Delta x$ & $10^{0}$             &   $5 (4$--$2)$& $5(5$--$2)$ & $ 5 (9$--$2)$& $ 6 (12$--$2)$    & $200(112$--$2)$ & $200(123$--$3)$\\ \hline
		$\Delta x$ &  $10^{2}$              &  $6 (4$--$2)$&$6(5$--$2)$ &$ 8 (8$--$2)$& $ 9 (9$--$2)$	& $6(22$--$2)$& $9(37$--$2)$\\ \hline
		$\Delta x$ & $10^{4}$             &    $6 (5$--$2)$&$6(5$--$2)$ &$ 9 (7$--$2)$& $ 9 (10$--$2)$	& $8(23$--$2)$& $200(65$--$4)$\\ \hline
		$\Delta x$ &  $10^{6}$              & $6 (5$--$2)$ &$6(5$--$2)$ & $ 9 (7$--$2)$& $ 9 (10$--$2)$	& $200(33$--$2)$& $200(92$--$3)$\\ \hline
		\multicolumn{8}{c}{}\\[-0.5em] \hline 
		$2\Delta x$ & $10^{-6}$              &  $4 (5$--$2)$ & $4(7$--$2)$ & $5 (11$--$2)$& $6(13$--$2)$  &$7(39$--$2)$ & $6(51$--$2)$\\ \hline
		$2\Delta x$ & $10^{-4}$              &  $4 (5$--$2)$ &$4(7$--$2)$ & $5 (11$--$2)$& $6(13$--$2)$& $7(41$--$2)$& $6(48$--$2)$\\ \hline
		$2\Delta x$ & $10^{-2}$            &   $4 (6$--$2)$&$4(7$--$2)$ & $ 5 (12$--$2)$  & $ 5 (13$--$2)$& $7(23$--$2)$ & $6(54$--$2)$\\ \hline
		$2\Delta x$ & $10^{0}$             &   $5 (4$--$2)$& $5(6$--$2)$ & $ 5 (9$--$2)$& $ 6 (11$--$2)$    & $7(27$--$2)$ & $200(107$--$3)$\\ \hline
		$2\Delta x$ & $10^{2}$              &  $6 (4$--$2)$&$6(5$--$2)$ &$ 8 (8$--$2)$& $ 8 (10$--$2)$	& $8(26$--$2)$& $9(37$--$2)$\\ \hline
		$2\Delta x$ & $10^{4}$             &    $6 (5$--$2)$&$6(5$--$2)$ &$ 8 (8$--$2)$& $ 9 (10$--$2)$	& $8(19$--$2)$& $9(41$--$2)$\\ \hline
		$2\Delta x$ & $10^{6}$              & $6 (5$--$2)$ &$6(5$--$2)$ & $ 8 (8$--$2)$& $ 9 (9$--$2)$	& $8(19$--$2)$& $9(41$--$2)$\\ \hline
		\multicolumn{8}{c}{}\\[-0.5em] \hline 
		$4\Delta x$ & $10^{-6}$              & $4(5$--$2)$ & $4(7$--$2)$ & $5(11$--$2)$ & $6(13$--$2)$ & $6(30$--$2)$ &  $200(126$--$6)$\\ \hline
		
		$4\Delta x$ &  $10^{-4}$              & $4(5$--$2)$ & $4(7$--$2)$ & $5(11$--$2)$ & $6(13$--$2)$ & $6(30$--$2)$  &  $200(98$--$4)$ \\ \hline
		
		$4\Delta x$ &  $10^{-2}$            & $4(5$--$2)$ & $4(7$--$2)$ & $5(12$--$2)$ & $6(13$--$2)$ & $6(30$--$2)$  &  $11(124$--$2)$  \\ \hline
		
		$4\Delta x$ &  $10^{0}$             & $4(5$--$2)$ & $4(6$--$2)$& $5(9$--$2)$ & $6(11$--$2)$ & $6(27$--$2)$  & $200(152-5)$  \\ \hline
		
		$4\Delta x$ &  $10^{2}$              & $6 (4$--$2)$ & $6(5$--$2)$ & $8(8$--$2)$ & $8(10$--$2)$ & $10(23$--$2)$ & $15(40-2)$\\ \hline
		
		$4\Delta x$ &  $10^{4}$             & $6 (4$--$2)$ & $6(5$--$2)$ & $8(8$--$2)$ & $8(10$--$2)$& $9(26$--$2)$ & $200(183$--$3)$\\ \hline
		
		$4\Delta x$ &  $10^{6}$              & $6 (4$--$2)$ & $6(5$--$2)$ & $8(8$--$2)$ &$8(10$--$2)$ & $9(26$--$2)$ & $200(45$--$2)$\\ \hline
		\multicolumn{8}{c}{}\\[-0.5em] \hline 
			\multicolumn{2}{|c|}{Sem. New.} 		& $4   $            &   $5$   & $10$ &    $13$	& $30$		& $44$ \\ \hline 
		
	\end{tabular}
	\caption{Test1: Number of outer iterations (maximum number - minimum number of inner iterations) for preconditioned generalized Newton varying $L$, $p$ and $\varepsilon$ and number of iterations for the semismooth Newton applied to \eqref{Mechelli_mini_13_Newton_system} (last row).}  \label{Mechelli_mini_13_Tabtest1}
	\vspace{-2em}
\end{table}
In this section, we study the behavior of the preconditioned generalized Newton method (Algorithm~\ref{Mechelli_mini_13_Alg1}) and its robustness against the Robin parameter $p$, the regularization $\varepsilon$ and the overlap $L$. It is well known that the convergence of the semismooth Newton method applied to \eqref{Mechelli_mini_13_Newton_system} deteriorates fast for decreasing values of $\varepsilon$, since the solution approaches the one of a pure pointwise state-constrained problem, whose adjoint variable $q$ lacks of $L^2$-regularity; cf. \cite{Mechelli_mini_13_Troeltzsch2010,Mechelli_mini_13_Mechelli2019}.
\begin{figure}[t]
	\centering 
	\includegraphics[height=30mm]{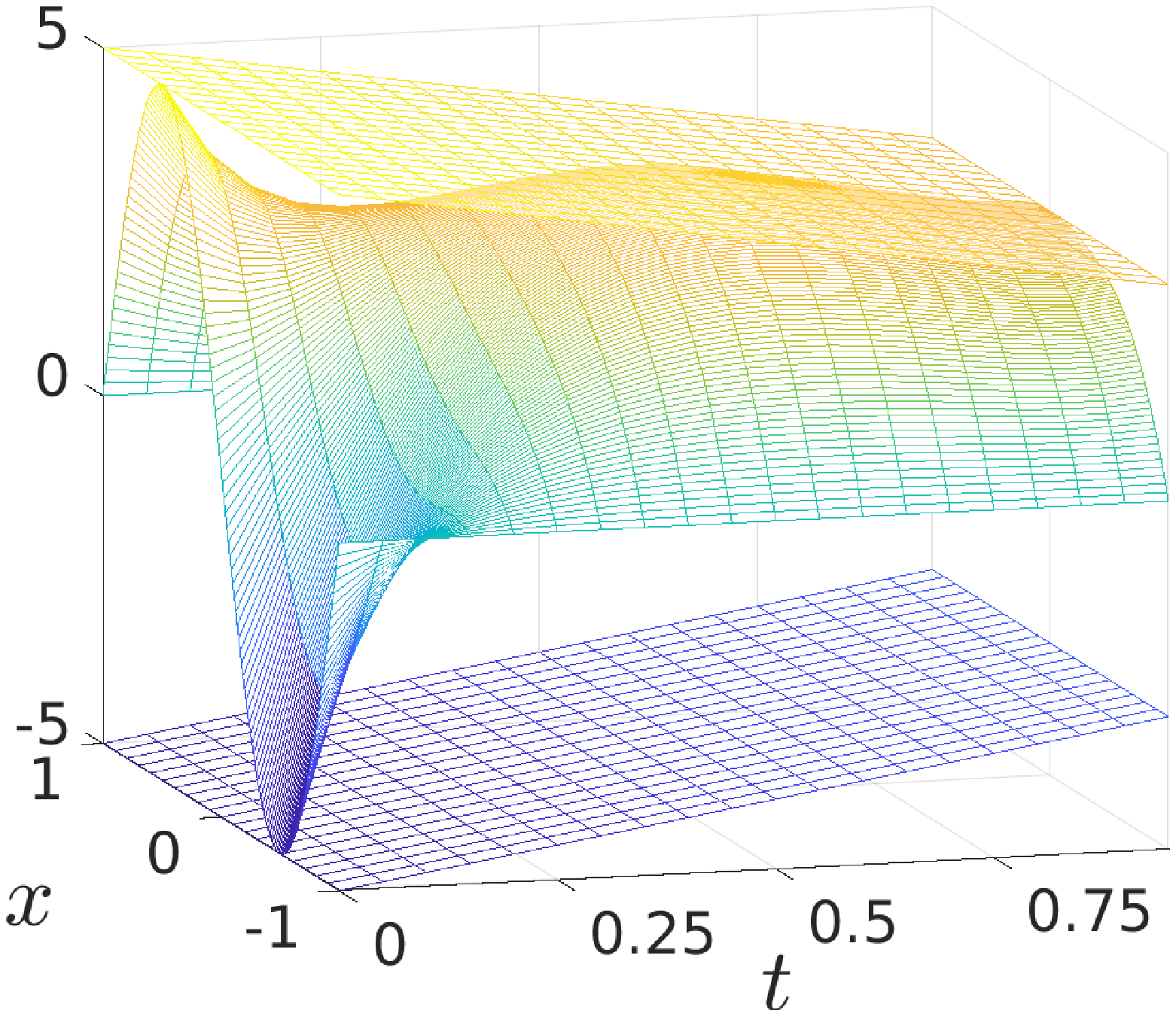} \hspace{1em}
	\includegraphics[height=30mm]{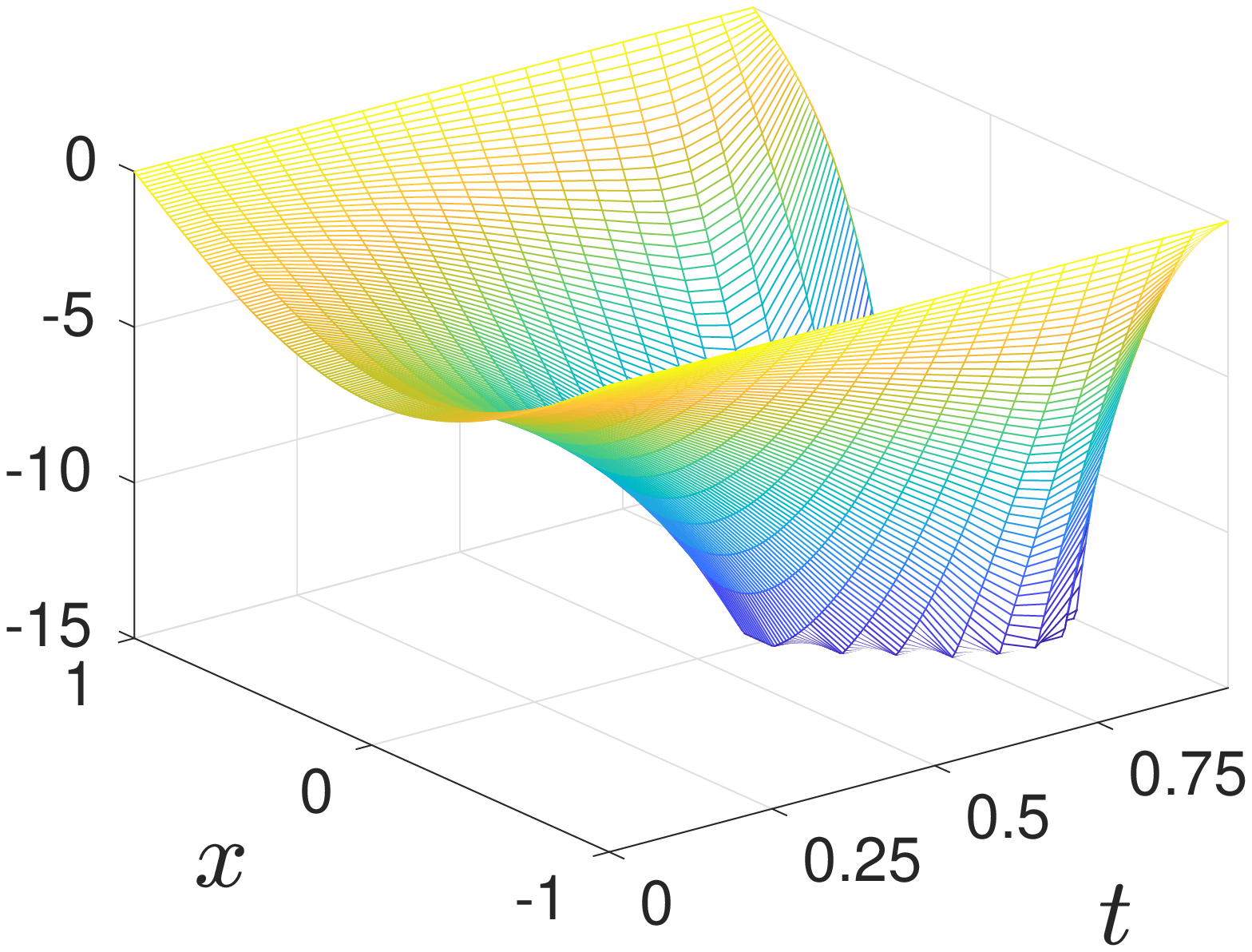}
	\caption{Test2: Optimal state with bound $c_y$ (left) and control (right) for $\varepsilon=5\times 10^{-4}$.} \label{Mechelli_mini_13_Fig2}
\end{figure}
\vspace{-0.5em}
\begin{table}[t]
	\centering
	\begin{tabular}{|c|c|c|c|c|c|c|c|}
		\hline
		$L$ & \diagbox{\hspace{1em}$p$}{$\varepsilon$\hspace{1em}} & $10^{-1}$ &$5\times 10^{-2}$ & $10^{-2}$ & $5\times 10^{-3}$ &$10^{-3}$ & $5\times 10^{-4}$ \\ \hline \multicolumn{8}{c}{}\\[-0.5em] \hline 
		$\Delta x$ &$10^{-6}$              &  $5 (5$--$2)$ & $6(7$--$2)$ & $10 (10$--$2)$& $200(61$--$2)$  & $200(102$--$2)$&  $200(297$--$4)$  \\ \hline
		$\Delta x$ &$10^{-4}$              &  $5 (5$--$2)$ &$6(7$--$2)$ & $10 (10$--$2)$& $200(32$--$2)$& $200(246$--$2)$   &  $200(145$--$2)$\\ \hline
		$\Delta x$ &$10^{-2}$            &   $5 (5$--$2)$&$6(7$--$2)$ & $ 8 (10$--$2)$  & $ 200 (25$--$2)$& $200(500$--$2)$ &  $200(500$--$4)$ \\ \hline
		$\Delta x$ &$10^{0}$             &   $5 (5$--$2)$& $6(6$--$2)$ & $ 6 (10$--$2)$& $ 9 (11$--$2)$    & $200(122$--$4)$ &  $200(193$--$2)$ \\ \hline
		$\Delta x$ &$10^{2}$              &  $6 (4$--$2)$&$7(5$--$2)$ &$ 9 (8$--$2)$& $ 9 (10$--$2)$	& $9(20$--$2)$     &    $10(25$--$2)$\\ \hline
		$\Delta x$ &$10^{4}$             &    $6 (4$--$2)$&$7(5$--$2)$ &$ 9 (8$--$2)$& $ 9 (11$--$2)$	& $11(20$--$2)$   &    $200(32$--$2)$\\ \hline
		$\Delta x$ &$10^{6}$              & $6 (4$--$2)$ &$7(5$--$2)$ & $ 9 (8$--$2)$& $ 9 (11$--$2)$	& $11(20$--$2)$  &     $200(67$--$4)$\\ \hline
		\multicolumn{8}{c}{}\\[-0.5em] \hline
		$2\Delta x$ & $10^{-6}$              &  $5 (6$--$2)$ & $6(7$--$2)$ & $12(11$--$2)$& $200(29$--$2)$  & $200(123$--$2)$ &  $200(206$--$3)$\\ \hline
		$2\Delta x$ &$10^{-4}$              &  $5 (6$--$2)$ &$6(7$--$2)$ & $12(11$--$2)$& $200(28$--$2)$& $200(91$--$2)$ &  $200(196$--$3)$\\ \hline
		$2\Delta x$ &$10^{-2}$            &   $5 (6$--$2)$&$6(7$--$2)$ & $11(11$--$2)$  & $200(25$--$2)$& $200(500$--$4)$ & $200(500$--$4)$ \\ \hline
		$2\Delta x$ &$10^{0}$             &   $5 (5$--$2)$& $6(6$--$2)$ & $ 6 (9$--$2)$& $ 7 (10$--$2)$    & $200(166$--$5)$ & $200(183$--$2)$\\ \hline
		$2\Delta x$ &$10^{2}$              &  $6 (4$--$2)$&$7(5$--$2)$ &$ 8 (8$--$2)$& $ 9 (11$--$2)$	& $9(20$--$2)$& $10(29$--$2)$\\ \hline
		$2\Delta x$ &$10^{4}$             &    $6 (4$--$2)$&$7(5$--$2)$ &$ 9 (7$--$2)$& $ 9 (11$--$2)$	& $10(20$--$2)$& $9(26$--$2)$\\ \hline
		$2\Delta x$ &$10^{6}$              & $6 (4$--$2)$ &$7(5$--$2)$ & $ 9 (7$--$2)$& $ 9 (11$--$2)$	& $10(19$--$2)$& $10(26$--$2)$\\ \hline
		\multicolumn{8}{c}{}\\[-0.5em] \hline
		$4\Delta x$ &$10^{-6}$              & $5(5$--$2)$ & $6(7$--$2)$ & $10(11$--$2)$  & $200(32$--$2)$ & $200(313$--$4)$ &  $200(187$--$4)$\\ \hline
		
		$4\Delta x$ &$10^{-4}$              & $5(5$--$2)$ & $6(7$--$2)$ & $10(11$--$2)$  & $200(27$--$2)$ & $200(145$--$4)$ &  $200(148$--$4)$\\ \hline
		
		$4\Delta x$ &$10^{-2}$              & $6(5$--$2)$ & $6(7$--$2)$ & $9(11$--$2)$  & $200(35$--$3)$ & $200(296-4)$ & $200(500-4)$ \\ \hline
		
		$4\Delta x$ &$10^{0}$               & $5(5$--$2)$ & $5(6$--$2)$& $6(8$--$2)$ & $8(11$--$2)$ & $200(136-3)$ & $200(500-3)$ \\ \hline
		
		$4\Delta x$ &$10^{2}$               & $6(4$--$2)$ & $7(5$--$2)$ & $6(8$--$2)$  & $8(11$--$2)$ & $11(20$--$2)$ & $14(44$--$2)$\\ \hline
		
		$4\Delta x$ &$10^{4}$               & $6(4$--$2)$ & $7(5$--$2)$ & $8(8$--$2)$ & $8(11$--$2)$& $10(20$--$2)$ & $12(26$--$2)$\\ \hline
		
		$4\Delta x$ &$10^{6}$               & $6(4$--$2)$ & $7(5$--$2)$ & $8(8$--$2)$ &$8(11$--$2)$ & $10(20$--$2)$ & $13(25$--$2)$\\ \hline
		\multicolumn{8}{c}{}\\[-0.5em] \hline
		\multicolumn{2}{|c|}{Sem. New.} 	& $4   $            &   $6$   & $10$ &    $12$	& $23$		& $30$ \\ \hline 
	\end{tabular}
	\caption{Test2: Number of outer iterations (maximum number - minimum number of inner iterations) for preconditioned generalized Newton varying $L$, $p$ and $\varepsilon$ and number of iterations for the semismooth Newton applied to \eqref{Mechelli_mini_13_Newton_system} (last row).}  \label{Mechelli_mini_13_Tabtest2} 
	\vspace{-2em}
\end{table}
The focus is on understanding if the WRM can be a valid (nonlinear) preconditioner and in which cases. We will perform two numerical experiments. In both tests we discretize the domain $\Omega$ with $n_x=161$ points and we apply a centered finite-difference scheme. Furthermore, we consider $n_t=21$ time discretization points and apply the implicit Euler method.  The initial guesses $y^{j,0}$ and $q^{j,0}$ are chosen randomly but feasible, i.e. such that $\left(\mathcal P(q^{j,0}(t,x)),-\varepsilon\mathcal{Q}^\varepsilon(y^{j,0}(t,x))\right)\in\mathcal{U}^\varepsilon_\mathsf{ad}$, since we noticed that choosing feasible initial guesses improves the convergence of the method. For the first test we choose $T=1$, $y_\circ(x)= 5\sin(\pi x)$, $f(t,x) = 20$, $c_u = 30$ and $c_y(t) = 10(1-t)+3$ for all $(t,x)\in(0,1)\times\Omega$. As one can see from Table~\ref{Mechelli_mini_13_Tabtest1}, for a decreasing $\varepsilon$ the number of iterations of the semismooth Newton method applied to \eqref{Mechelli_mini_13_Newton_system} increases and its convergence deteriorates fast. On the contrary, the number of iterations of Algorithm~\ref{Mechelli_mini_13_Alg1} is almost constant as $\varepsilon$ varies (when it converges). Choosing $p=10^2$ guarantees that the method is convergent for any choice of $\varepsilon$ and $L$. In particular, for small $\varepsilon$, such as $10^{-3}$ and $5\times 10^{-4}$, the speed-up in terms of number of iterations is also significant. According to Table~\ref{Mechelli_mini_13_Tabtest1}, there are some combinations for which Algorithm~\ref{Mechelli_mini_13_Alg1} reaches a maximum number of iterations. This issue can be related to the fact that $y^{j,k}$ and $q^{j,k}$ might become unfeasible during Algorithm~\ref{Mechelli_mini_13_Alg1} and when traced to the interface of the other subdomain might cause oscillations. A more detailed study on this convergence issue and on possible solutions is beyond the scope of this short manuscript and will be investigated in a future work. For the second test we choose $T=1$, $y_\circ(x)= 5\sin(\pi x)$, $f(t,x) = 18$, $c_u = 15$ and $c_y(t)= 2(1-t)+3$ for $(t,x)\in(0,1)\times\Omega$. In this case, there are more points in the space-time domain for which both bounds become active (cf. Figures~\ref{Mechelli_mini_13_Fig1}-\ref{Mechelli_mini_13_Fig2}). This makes the problem even more difficult to be solved by the WRM, since its nonlinearities are more strongly activated. In Table~\ref{Mechelli_mini_13_Tabtest2}, in fact, the number of cases for which Algorithm~\ref{Mechelli_mini_13_Alg1} does not converge increases with respect to the first numerical experiment, particularly for $\varepsilon$ small. Increasing the size of the overlap helps when $p$ is large enough, i.e. when the Dirichlet part of the transmission conditions of the WRM dominates the Neumann part. Transmission conditions of Dirichlet type and large-enough overlap guarantee that the number of unfeasible points at the interface is significantly reduced, so that Algorithm~\ref{Mechelli_mini_13_Alg1} converges more easily. This confirms the previous remark on the importance of the feasibility of the iterations. Note that, also in the second test, there always exists a combination of $p$ and $L$ for which Algorithm~\ref{Mechelli_mini_13_Alg1} is faster than the semismooth Newton method, in particular for a small $\varepsilon$. Therefore, the WRM is a valid preconditioner in order to solve \eqref{Mechelli_mini_13_Newton_system}, although some issues have to be still clarified. These will be the focus of a future work.

	\vspace{-2em}
	\bibliographystyle{plain}
	\bibliography{Mechelli_mini_13}

\begin{thebibliography}{10}

\bibitem{Mechelli_mini_13_Benamou1996}
J.-D. Benamou.
\newblock A domain decomposition method with coupled transmission conditions
  for the optimal control of systems governed by elliptic partial differential
  equations.
\newblock {\em SIAM J. Numer. Anal.}, 33(6):2401--2416, 1996.

\bibitem{Mechelli_mini_13_Dolean2016}
V.~Dolean, M.~J. Gander, W.~Kheriji, F.~Kwok, and R.~Masson.
\newblock Nonlinear preconditioning: How to use a nonlinear {S}chwarz method to
  precondition {N}ewton's method.
\newblock {\em SIAM J. Sci. Comput.}, 38(6):A3357--A3380, 2016.

\bibitem{Mechelli_mini_13_GanderHalpern1}
M.~J. Gander and L.~Halpern.
\newblock Optimized {S}chwarz waveform relaxation methods for advection
  reaction diffusion problems.
\newblock {\em SIAM J. Numer. Anal.}, 45(2):666--697, 2007.

\bibitem{Mechelli_mini_13_Heink2006}
M.~Heinkenschloss and H.~Nguyen.
\newblock {N}eumann--{N}eumann domain decomposition preconditioners for
  linear-quadratic elliptic optimal control problems.
\newblock {\em SIAM J. Sci. Comput.}, 28(3):1001--1028, 2006.

\bibitem{Mechelli_mini_13_Hintermueller2002}
M.~Hinterm\"uller, K.~Ito, and K.~Kunisch.
\newblock The primal-dual active set strategy as a semismooth {N}ewton method.
\newblock {\em SIAM J. Optim.}, 13(3):865--888, 2002.

\bibitem{Mechelli_mini_13_Krumbiegel2009}
K.~Krumbiegel and A.~R{\"o}sch.
\newblock A virtual control concept for state constrained optimal control
  problems.
\newblock {\em Comput. Optim. Appl.}, 43:213–233, 2009.

\bibitem{Mechelli_mini_13_LionsMagenes1}
J.~L. Lions and E.~Magenes.
\newblock {\em Non-homogeneous boundary value problems and applications (Vol
  II)}.
\newblock Die Grundlehren der mathematischen Wissenschaften. Springer-Verlag
  Berlin Heidelberg, 1972.

\bibitem{Mechelli_mini_13_Mechelli2019}
L.~Mechelli.
\newblock {\em {POD}-based state-constrained economic Model Predictive Control
  of convection-diffusion phenomena}.
\newblock PhD thesis, University of Konstanz, 2019.

\bibitem{Mechelli_mini_13_Mechelli2018}
L.~Mechelli and S.~Volkwein.
\newblock {POD}-based economic optimal control of heat-convection phenomena.
\newblock In M.~Falcone, R.~Ferretti, L.~Gr{\"u}ne, and W.~M. McEneaney,
  editors, {\em Numerical Methods for Optimal Control Problems}, pages 63--87,
  Cham, 2018. Springer International Publishing.

\bibitem{Mechelli_mini_13_Troeltzsch2010}
F.~Tr\"oltzsch.
\newblock {\em Optimal Control of Partial Differential Equations: Theory,
  Methods and Applications}.
\newblock American Mathematical Society, 2010.

\end{thebibliography}
\end{document}